\iffalse
\documentclass{mcom-l}
\usepackage{amssymb}
\usepackage{booktabs}
\usepackage[ruled,vlined,linesnumbered]{algorithm2e}
\usepackage{comment}
\newtheorem{thm}{Theorem}
\newtheorem{lem}{Lemma}

\begin{document}

\title[A proof of the conjecture of Cohen and Mullen]{A proof of the conjecture of Cohen and Mullen on sums of primitive roots}

\author{Stephen D. Cohen}
\address{School of Mathematics and Statistics, University of Glasgow, Scotland}
\email{Stephen.Cohen@glasgow.ac.uk}
\curraddr{}
\thanks{}

\author{Tom\'{a}s Oliveira e Silva}
\address{Departamento de Electr{\'o}nica, Telecomunica{\c c}{\~o}es e Inform{\'a}tica / IEETA,
  Universidade de Aveiro, Portugal}
\email{tos@ua.pt}
\urladdr{http://www.ieeta.pt/~tos}
\curraddr{}
\thanks{}

\author{Tim Trudgian}
\address{Mathematical Sciences Institute, The Australian National University, ACT 0200, Australia}
\email{timothy.trudgian@anu.edu.au}
\curraddr{}
\thanks{Supported by Australian Research Council DECRA Grant DE120100173}



\dedicatory{}

\begin{abstract}
  We prove that for all $q>61$, every non-zero element in the finite field $\mathbb{F}_{q}$ can be
  written as a linear combination of two primitive roots of $\mathbb{F}_{q}$. This resolves a
  conjecture posed by Cohen and Mullen.
\end{abstract}

\maketitle

\else
\documentclass[11pt]{article}
\usepackage{a4wide}

\usepackage{booktabs}
\usepackage[ruled,vlined,linesnumbered]{algorithm2e}

\usepackage{amsthm}
\usepackage{amsmath}
\usepackage{amssymb}
\newtheorem{thm}{Theorem}
\newtheorem{lem}{Lemma}

\title{A proof of the conjecture of Cohen and Mullen on sums of primitive roots}

\author{
  Stephen D. Cohen \\
  School of Mathematics and Statistics, \\
  University of Glasgow, Scotland \\
  Stephen.Cohen@glasgow.ac.uk
\and
  Tom\'{a}s Oliveira e Silva \\
  Departamento de Electr{\'o}nica, Telecomunica{\c c}{\~o}es e Inform{\'a}tica / IEETA \\
  University of Aveiro, Portugal \\
  tos@ua.pt
\and
  Tim Trudgian\footnote{Supported by Australian Research Council DECRA Grant DE120100173.} \\
  Mathematical Sciences Institute \\
  The Australian National University, ACT 0200, Australia \\
  timothy.trudgian@anu.edu.au
}
\date{February 9, 2014}

\begin{document}

\maketitle

\begin{abstract}
  \noindent
  We prove that for all $q>61$, every non-zero element in the finite field $\mathbb{F}_{q}$ can be
  written as a linear combination of two primitive roots of $\mathbb{F}_{q}$. This resolves a
  conjecture posed by Cohen and Mullen.
\end{abstract}


\fi

\section{Introduction}

For $q$ a prime power, let $\mathbb{F}_{q}$ denote the finite field of order~$q$, and let
$g_{1}$, $g_{2}$, $\ldots\,$, $g_{\phi(q-1)}$ denote the primitive roots of~$q$. Various questions
have been asked about whether non-zero elements of $\mathbb{F}_{q}$ can be written as a linear sum
of two primitive roots, $g_{1}$ and $g_{2}$. To develop this idea, let $a$, $b$ and $c$ be
arbitrary non-zero elements in~$\mathbb{F}_{q}$. Is there some $q_{0}$ such that there is always
one representation
\begin{equation}\label{e:vc}
  a = bg_{n} + c g_{m}
\end{equation}
for all $q>q_{0}$? Since such a representation is possible if and only if $a/b=g_{n}+c/b\,g_{m}$,
we may suppose that $b=1$ in~(\ref{e:vc}).
Accordingly, define $\mathcal {G}$ to be the set of prime powers $q$ such that for all non-zero
$a, c \in \mathbb{F}_{q}$ there exists a primitive root $g \in \mathbb{F}_{q}$ such that $a-cg$ is
also a primitive root of $\mathbb{F}_{q}$.

It appears that Vegh~\cite{Vegh} was the first to consider a specific form of~(\ref{e:vc}),
namely, that with $b=1$ and $c=-1$. This has been referred to as Vegh's Conjecture ---
see~\cite[\S F9]{Guybook}. Vegh verified his own conjecture for $61 < q <2000$;
Szalay~\cite{Szalay} proved it for $q>q_{0}$ and claimed that one could take $q_{0}=10^{19}$. In
the special case when $a=1$, Cohen~\cite{Cohen} proved Vegh's conjecture for all~$q>7$.

Golomb~\cite{Golomb} proposed~(\ref{e:vc}) with $b=1$ and~$c=1$. This has applications to Costas
arrays, which appear in the study of radar and sonar signals. This was proved by Sun~\cite{Sun}
for $q>2^{60}\approx 1.15\times 10^{18}$.

Cohen and Mullen~\cite{CohMul} considered~(\ref{e:vc}) in its most general form, namely $b=1$ and
arbitrary non-zero $c$ and~$a$. Cohen~\cite{Cohen1993} calls this `Conjecture~G'. Cohen and Mullen
proved Conjecture~G for all $q \geq 4.79\times 10^{8}$; Cohen~\cite{Cohen1993} proved it for all
$q\geq 3.854\times 10^{7}$, and states that it is true for even~$q>4$; it is false for $q=4$.
Chou, Mullen, Shiue and Sun~\cite{Chou}
tested it for odd $q<2130$ and found that it failed only for $q=3$, $5$, $7$, $11$, $13$, $19$,
$31$, $43$, and $61$.  Thus, in effect, Conjecture~G can be interpreted as claiming that all prime
powers exceeding $61$ lie in the set~$\mathcal{G}$. What this means is that `all' one needs to do
is to check~(\ref{e:vc}) for $2130\leq q \leq 3.854\times 10^{7}$.

We improve on Cohen's method, given in~\cite{Cohen1993}, to isolate easily the possible
counterexamples to Conjecture~G. We compile an initial list of values of $q$ that may need
checking. We then examine this list in more detail, sieving out some values of~$q$. This produces
a secondary list of only $777$ values of $q$ with $2131\leq q \leq 2762761$. This list is
just small enough to enable us to verify~(\ref{e:vc}) for each~$q$. The result is:

\begin{thm}\label{t:thm1}
  For $q>61$ and for arbitrary non-zero elements $a$, $b$, $c$ of $\mathbb{F}_{q}$, there is
  always one representation of the form
  $$
    a = bg_{n} + c g_{m},
  $$
  where $g_{n}$ and $g_{m}$ are primitive roots of $\mathbb{F}_{q}$.
\end{thm}

\section{Theory}

Let $\omega(n)$ denote the number of distinct prime factors of $n$ so that $W(n)=2^{\omega(n)}$ is the
number of square-free divisors of~$n$. Also, let $\theta(n)=\prod_{p|n}(1-p^{-1})$. From the
working of~\cite{Cohen1993} (see also~\cite{CohMul}) one can conclude that a prime power
$q \in \mathcal{G}$ if $q > W(q-1)^4$ and hence if $\omega(q-1) \geq 16$ or $q>2^{60}$. More
significantly, a sieving method was given yielding improved lower bounds for $q$ guaranteeing
membership of~$\mathcal{G}$. Instead of $W(q-1)$, these depend on appropriate choices of divisors
$e_1,e_2$ of $q-1$ and the quantities $W(e_i)$ and $\theta_{e_i}$, $i=1,2$, and can be applied to
successively smaller values of $\omega(q-1) \leq 15$. In particular, it was shown that, if
$\omega(q-1) \geq 9$, then $q \in\mathcal{G}$. Further, for each value of $\omega(q-1) \leq 8$ an upper
bound can be derived on the set of prime power values requiring further analysis.

It turns out that the lists of possible exceptions that are thereby obtained from the method
of~\cite{Cohen1993} are small enough for direct computer verification on contemporary computer
hardware. However, we can do better, as we now proceed to show. For any integer $n$ define its
radical $\mathrm{Rad}(n)$ as the product of all distinct prime factors of~$n$. In the appendix we prove

\begin{thm}\label{t:Cohen}
  Let $q\geq 4 $ be a prime power. Let $e$ be a divisor of $q-1$.  If $\mathrm{Rad}(e)=\mathrm{Rad}(q-1)$
  set $s=0$ and $\delta=1$.  Otherwise,  let $p_1,\ldots,p_s$, $s\geq 1$, be the primes dividing
  $q-1$ but not~$e$ and set $\delta=1-2\sum_{i=1}^s p_i^{-1}$. Suppose that $\delta$
  is positive and that
  \begin{equation}\label{e:Cohen}
    q>\left(\frac{2s-1}{\delta}+2\right)^2 W(e)^4.
  \end{equation}
  Then $q\in\mathcal{G}$.
\end{thm}

As an example of the usefulness of this theorem, consider the case
$\omega(q-1)=8$. For $s=5$ we have $W(e)=8$ and
$\delta\geq 1-2\bigl(\frac{1}{7}+\frac{1}{11}+\frac{1}{13}+\frac{1}{17}+\frac{1}{19}\bigr)$.
Therefore the right hand side of~(\ref{e:Cohen}) is at most $14647129.006$, and so $q>14647129$
guarantees membership in~$\mathcal{G}$ when $\omega(q-1)=8$. Moreover, up to $14647129$ there is only
one prime power with $\omega(q-1)=8$, namely
$q=13123111=2\cdot 3\cdot 5\cdot 7\cdot 11\cdot 13\cdot 19\cdot 23+1$. Using again
Theorem~\ref{t:Cohen}, still with $s=5$ but this time with
$\delta=1-2\bigl(\frac{1}{7}+\frac{1}{11}+\frac{1}{13}+\frac{1}{19}+\frac{1}{23}\bigr)$ tailored
to the specific value of~$q$, allows us to conclude that $13123111\in\mathcal{G}$.

We repeat the procedure for $1\leq \omega(q-1)\leq 7$, and $q\geq 2130$, each time noting the upper
bound of the intervals that need further analysis. These are reported in the second column in
Table~\ref{t:bounds}. We then enumerate all possible $q$ and eliminate as many values as we can
by checking if~(\ref{e:Cohen}) is true for some value of~$s$. We are left with a final list of
values of $q$ that need checking. Column~3 of Table~\ref{t:bounds} contains the largest elements
of these lists; Columns~4 and~5 contain respectively the initial and final number of elements of
these lists, discriminating primes (on the left of the summation sign) and prime powers (on the
right of the summation sign).

\begin{table}[ht]
  \caption{Improved bounds for $q$}
  \label{t:bounds}
  \centering
  \begin{tabular}{ccccc}
    \hline\hline
    $\omega(q-1)$ & Upper bound & Largest $q$ & Initial list size & Final list size \\
    \hline
    $8$ & $14647129$ &         --- &    $1+0$  &   $0+0$  \\
    $7$ &  $3402711$ &   $2762761$ &   $78+1$  &  $22+1$  \\
    $6$ &   $947062$ &    $840841$ &  $635+6$  & $162+4$  \\
    $5$ &   $238715$ &    $231001$ & $1741+21$ & $290+9$  \\
    $4$ &    $34124$ &     $33601$ & $1024+24$ & $259+10$ \\
    $3$ &     $3441$ &      $4057$ &   $84+5$  &  $16+4$  \\
    \hline\hline
  \end{tabular}
\end{table}

It is worthwhile to mention that the largest value that needs to be verified ($2762761$) is
considerably smaller than the largest value that, according to~\cite{Cohen1993}, would have to
be verified ($25555531$).

\section{Computation}\label{sec3}

Let $q=p^k$, where $p$ is a prime. When $k=1$ the full finite field machinery is not needed to
test Conjecture~G. For efficiency reasons, we thus developed two programs to test it: one
to deal with the case $k=1$, and another to deal with the case $k>1$.

\subsection{\boldmath Verification of Conjecture~G when $q=p$}

One way to verify Conjecture~G for a given value of $p$ is to call
Algorithm~\ref{a:p1} (or Algorithm~\ref{a:p2}) for $c=1,2,\ldots,p-1$. If it returns success in
all cases then Conjecture~G is true. Otherwise it is false.

\begin{algorithm}[H]
  \DontPrintSemicolon
  \caption{Verification that for every non-zero element $a$ of $\mathbb{F}_p$ there exist two
           primitive roots also of $\mathbb{F}_p$, $g_m$ and $g_n$, such that $a=g_n+cg_m$.
           \label{a:p1}}
  Set $t_0$ to $1$, set $t_1,t_2,\ldots,t_{p-1}$ to $0$, and set $r$ to $p-1$ \;
  \For{$m=1,2,\ldots,\phi(p-1)$}
  {
    Set $d$ to $cg_m$ \;
    \For{$n=1,2,\ldots,\phi(p-1)$}
    {
      Set $a$ to $g_n+d$ \;
      \lIf{$t_a$ is equal to $0$}{set $t_a$ to $1$ and decrease $r$}
    }
    \lIf{$r$ is equal to $0$}{terminate with success}
  }
  Terminate with failure \;
\end{algorithm}

If at some point during the execution of Algorithm~\ref{a:p1} $t_a$ is equal to $0$ then
no solution to $a=g_n+cg_m$ was found up to that point. Note that $r$ counts the number of
$t_a\!$'s that are still equal to~$0$. For small values of $m$, Algorithm~\ref{a:p1} is
quite efficient at discarding values of~$a$, but when $r$ starts to become much smaller than
$\phi(p-1)$ it becomes inefficient. Algorithm~\ref{a:p2} handles small values of $r$ better,
but is less efficient than Algorithm~\ref{a:p1} when $r$ is large.

\begin{algorithm}[H]
  \DontPrintSemicolon
  \caption{Verification that for every non-zero element $a$ of $\mathbb{F}_p$ there exist two
           primitive roots also of $\mathbb{F}_p$, $g_m$ and $g_n$, such that $a=g_n+cg_m$.
           \label{a:p2}}
  Set $a_1$ to $1$, $a_2$ to $2$, $\ldots\,$, $a_{p-1}$ to $p-1$, and set $r$ to $p-1$ \;
  \For{$m=1,2,\ldots,\phi(p-1)$}
  {
    Set $i$ to $1$ and set $d$ to $cg_m$ \;
    \While{$i\leq r$}
    {
      Set $j$ to $a_i-d$ \;
      \eIf{$j$ is a primitive root}{Set $a_i$ to $a_r$ and decrease $r$}{Increase $i$}
    }
    \lIf{$r$ is equal to $0$}{terminate with success}
  }
  Terminate with failure \;
\end{algorithm}

Note that $a_1,\ldots,a_r$ hold the $r$ values of $a$ for which no solution of $a=g_n+cg_m$ has
yet been found. It is quite easy to switch from the first algorithm to the second and
\textit{vice versa} at the point where a new value of $m$ is to be considered. We leave the
easy details about how to do this for the reader to amuse herself/himself. In our implementation
of these algorithms, we switch from the first to the second as soon as $r$ drops below
$0.25\,\phi(p-1)$. For our range of values of $p$, the hybrid algorithm has an execution time that
is very nearly proportional to~$p$.

The verification of Conjecture~G for a given value of $p$ can be done easily in parallel by
assigning different ranges of values of $c$ (a work unit) to each of the available processor
cores. This was done for the $749$ prime values of $q$ that, according to Table~\ref{t:bounds} had
to be tested. Using an Intel Core 2 Duo E8400 processor running at $3.0$~GHz this took about
$12.4$ one-core days. In all cases it was found that~$q\in\mathcal{G}$. A second run of the
program, on an Intel Core i5-2400 processor running at $3.1$~GHz, produced exactly the same
results for each work unit, with the obvious exception of execution times, and required $13.4$
one-core days. (This second run was slower due to data cache effects.) For each work unit we
recorded the number of times the hybrid algorithm terminated with a given value of $m$ and we
computed a $32$-bit cyclic redundancy checksum that depended on the values of some variables at
key points of the hybrid algorithm.

For our first run, it took approximately $1.3\times10^{-8}\,p^2$ seconds to
test Conjecture~G for a given value of $p$. For $p$ above $10^3$ our
algorithm terminated with success for an average value of $m$ that was close to
$\log\bigl(\frac{1}{2p}\bigr)/\log\bigl(1-\frac{\phi(p-1)}{p}\bigr)$.

To double-check the results of~\cite{Chou}, we also ran our programs for all primes up to~$2130$.
As expected, we found that Conjecture~G is false only for $q=3$, $q=5$, $q=7$, $q=11$, $q=13$,
$q=19$, $q=31$, $q=43$, and $q=61$.

\smallskip
While the two runs of the program were underway, we found a way to share most of the work needed
to test several values of $c$, thus giving rise to a much more efficient program. The key to this
improvement is the observation that in~(\ref{e:vc}) $c$ appears multiplied by $g_m$. Thus, if
instead of iterating on $m$ on the outer loops of Algorithms~\ref{a:p1} and \ref{a:p2} we iterate
on carefully chosen values of the product $cg_m$, which we denote by $d$, then it becomes possible
to exclude the same value of $a$ simultaneously for several different~$c$'s, also carefully
chosen.

To explain how this is done, let $g$ be one primitive root of~$\mathbb{F}_p$, let $u$ be the
largest non-repeated prime factor of $p-1$, and let $v=(p-1)/u$. The set
$G=\{\,g^{1+iv}\,\}_{i=0}^{u-1}$ contains exactly $u-1$ primitive roots and exactly one
non-primitive root, which will be denoted by~$z$. Since $g$ is a primitive root the sets
$C_o=\{\,g^{o+iv}\,\}_{i=0}^{u-1}$, $0\leq o\leq v-1$, are pairwise disjoint and their union is
the set of the non-zero elements of~$\mathbb{F}_p$. Moreover, the set formed by the products of
one member of $G$ and one member of $C_o$ is $C_{o+1}$ (note that $C_p=C_0$). Let $C'_o$ be a non
empty proper subset of~$C_o$, and let $C''_{o+1}$ be the corresponding subset of $C_{o+1}$ whose
members are obtained by multiplying the members of $C'_o$ by the non-primitive root~$z$. To test
simultaneously the values of $c$ belonging to $C'_o$, we use as values of $d$ the complement of
$C''_{o+1}$, i.e., the set $C_{o+1}-C''_{o+1}$. This ensures, for every $c\in C'_o$, that $d$ is
the product of $c$ and a primitive root of~$G$. These observations give rise to
Algorithm~\ref{a:p3}.

\begin{algorithm}[H]
  \DontPrintSemicolon
  \caption{Efficient verification that for $c\in C'_o$ and for every non-zero element $a$ of
           $\mathbb{F}_p$ there exist two primitive roots also of $\mathbb{F}_p$, $g_m$ and $g_n$,
           such that $a=g_n+cg_m$.
           \label{a:p3}}
  Set $t_0$ to $1$, set $t_1,t_2,\ldots,t_{p-1}$ to $0$, and set $r$ to $p-1$ \;
  \For{$d$ belonging to the complement of $C''_{o+1}$}
  {
    \For{$n=1,2,\ldots,\phi(p-1)$}
    {
      Set $a$ to $g_n+d$ \;
      \lIf{$t_a$ is equal to $0$}{set $t_a$ to $1$ and decrease $r$}
    }
    \lIf{$r$ is equal to $0$}{terminate with success}
  }
  \For{$c\in C'_o$}
  {
    Run Algorithm~\ref{a:p1} with a copy of the $t_a$ and $r$ variables, beginning it at
      line~$2$ \;
    Terminate with failure if Algorithm~\ref{a:p1} failed \;
  }
  Terminate with success \;
\end{algorithm}

To test Conjecture~G for $c\in C_o$ it is obviously necessary to run Algorithm~\ref{a:p3}
twice: once for $C'_o$ and once more with $C'_o$ replaced by its complement. For that reason, to
make the entire testing effort more efficient, $C'_o$ and its complement should have approximately
the same number of members (as $u$ is usually odd, one should have one more member than the
other). As before, the verification of Conjecture~G for a given value of $p$ can easily be
done in parallel, now by assigning different ranges of values of $o$ to each of the available
processor cores. This was done for the $749$ primes values of $q$ that had to be tested. Taking only $2.7$ one-core days plus $1.9$ days for double-checking, this computation
confirmed the results of our first two runs.

\subsection{\boldmath Verification of Conjecture~G when $q=p^k$}

We have chosen to represent a generic element $a$ of $\mathbb{F}_q$ by the polynomial
$\sum_{i=0}^{k-1} a_k x^k$ of formal degree $k-1$ with coefficients in $\mathbb{F}_p$, and
henceforth to do multiplications in $\mathbb{F}_q$ using polynomial arithmetic modulo a monic
irreducible polynomial of degree~$k$. Since all finite fields with $q$ elements are isomorphic to
each other, any irreducible polynomial will do. Since we also need a primitive root, instead of
finding first an irreducible polynomial and then finding a primitive root for that particular
model of the finite field, we fix the primitive root (for convenience we have chosen $g=x$), and
then find a monic polynomial of degree $k$ for which $g^{q-1}=1$ and for which $g^{(q-1)/f}\neq 1$
for each prime factor $f$ of~$q-1$. This ensures that the polynomial is indeed primitive and that
$g$ is one of its primitive roots.

Although the arithmetic operations are different when $q=p^k$, the main ideas of the three
algorithms presented above remain valid. In all three algorithms it is necessary to replace $p$
by~$q$. In addition, in Algorithm~\ref{a:p1} it is necessary to replace in line~$6$ $t_a$ by
$t_{a'}$, where $a'$ is the value of the polynomial that represents $a$ for $x=p$, because $a$ was
being used there as an index. Likewise for Algorithm~\ref{a:p3} in line~$5$. In
Algorithm~\ref{a:p2} it is necessary to replace the way the $a_j$ variables are initialised, since these are now 
polynomials with coefficients $a_{ji}$.

Using Algorithm~\ref{a:p2}, it took $7.0$~days on a single core of a $2.8$~GHz processor, plus
$11.9$ days to double-check the results on a slower processor, to check the conjecture for the
$28$ prime powers that had to be tested. In all cases it was found that~$q\in\mathcal{G}$.

Finally, to double check the results of~\cite{Chou}, we also ran our programs for all prime powers
up to~$2130$. As expected, we found that Conjecture~G is false only for $q=4$.

\appendix

\section{Proof of Theorem~\ref{t:Cohen}}

Sieving methods for problems involving primitive roots have been refined since those described in
\cite{CohMul} and \cite{Cohen1993} were formulated. A recent illustrative example occurs
in ~\cite{CohHuc} and is a model for the line of argument pursued here.

Throughout, suppose that $a$ and $b$ are arbitrary given non-zero members of $\mathbb{F}_{q}$.
For any $g \in \mathbb{F}_{q}$ set $a-cg=g^*$.

Let $e$ be a divisor of $q-1$. Call $g \in \mathbb{F}_{q}$ \emph{$e$-free} if $g \neq 0$ and
$g = h^d$, where $h \in \mathbb{F}_{q}$ and $d|e$, implies $d=1$. The notion of $e$-free depends
(among divisors of $q-1$) only on $\mathrm{Rad}(e)$. Moreover, in this terminology a primitive
root of $\mathbb{F}_{q}$ is a $(q-1)$-free element.
Next, given divisors $e_1, e_2$ of $q-1$, define $N(e_1,e_2)$ to be the number of
$g \in \mathbb{F}_{q}$ such that $g$ is $e_1$-free and $g^*$ is $e_2$-free. In order to show that
a prime power $q \in \mathcal{G}$ we have to show that $N(q-1,q-1)$ is positive (for every choice
of $a$  and $c$). The value of $N(e_1,e_2)$ can be expressed explicitly in terms of Jacobi sums over
$\mathbb{F}_{q}$ as follows. We have
\begin{equation}\label{jac}
  N(e_1,e_2)= \theta(e_1)\theta(e_2) \int_{d_1|e_1} \int_{d_2|e_2}
    \chi_{d_1}(1/c)(\chi_{d_1}\chi_{d_2})(a)J(\chi_{d_1},\chi_{d_2}).
\end{equation}
Here, for a divisor $e$ of $q-1$,
\[
  \int_{d|e} = \sum_{d|e}\frac{\mu(d)}{\phi(d)} \sum_{\chi_{d}},
\]
where the sum over $\chi_{d}$ is the sum over all $\phi(d)$ multiplicative characters $\chi_d$ of
$\mathbb{F}_{q}$ of exact order $d$, and $J(\chi_{d_1},\chi_{d_2})$ is the Jacobi sum
$\sum_{g \in \mathbb{F}_{q}} \chi_{d_1}(g)\chi_{d_2}(1-g)$.
(All multiplicative characters on $\mathbb{F}_q$ by convention take the value $0$ at $0$.)

Next, we present a combinatorial sieve. Let $e$ be a divisor of $q-1$. In practice, this
\emph{kernel} $e$ will be chosen such that  $\mathrm{Rad}(e)$ is the product of the smallest
primes in $q-1$.  Use the notation of   Theorem~\ref{t:Cohen}. In particular, if $\mathrm{Rad(e)} < \mathrm{Rad}(q-1)$
 let $p_1, \ldots, p_s$, $s \geq 1$, be the primes dividing $q-1$ but not $e$ and
 set $\delta=1-\sum_{i=1}^s 2p_i^{-1}$.
In practice, it is essential to choose $e$ so that $\delta >0$.
\begin{lem}\label{sieve}
  Suppose $e$ is a divisor of $q-1$. Then, in the above notation,
  \begin{equation}\label{sieveeq1}
    N(q-1,q-1) \geq \sum_{i=1}^sN(p_ie,e)+ \sum_{i=1}^sN(e,p_ie)-(2s-1)N(e,e).
  \end{equation}
  Hence
  \begin{equation}\label{sieveeq2}
    N(q-1,q-1) \geq \sum_{i=1}^s\{[N(p_ie,e)-\theta(p_i)N(e,e)]+ [N(e,p_ie)-\theta(p_i)N(e,e)]\}+
    \delta N(e,e).
  \end{equation}
\end{lem}

\begin{proof} 
The various $N$ terms on the right side of~(\ref{sieveeq1}) can be regarded as counting
  functions on the set of $g \in \mathbb{F}_{q}$ for which both $g$ and $g*$ are $e$-free.
  In particular, $N(e,e)$ counts all such elements, whereas, for example, $N(p_ie,e), \ i \leq s$,
  counts only those for which additionally $g$ is $p_i$-free. Since $N(q-1,q-1)$ is the number of
  $e$-free elements $g$ for which $g$ and $g^*$ are both
  $p_i$-free for every $i \leq s$, we see that, for a given $e$-free $g \in \mathbb{F}_{q}$, the right side
  of~(\ref{sieveeq1}) clocks up $1$ if $g$ and $g^*$ are both primitive and otherwise contributes
  a non-positive (integral) quantity. This establishes (\ref{sieveeq1}).
  Since $\theta(p_i)=1-1/p_i$, the bound~(\ref{sieveeq2}) is deduced simply by rearranging the
  right side of~(\ref{sieveeq1}).
\end{proof}

\begin{lem}\label{t:N}
  Suppose that $q \geq 4$ is a prime power and $e$ is a divisor of $q-1$. Then
  \begin{equation}\label{Neq1}
    N(e,e) \geq \theta(e)^{2}(q-W(e)^2\sqrt{q}).
  \end{equation}
  Moreover, for any prime $l$ dividing $q-1$ but not $e$, we have
  \begin{equation}\label{Neq2}
    |N(le,e) -\theta(l)N(e,e)| \leq (1-1/l)\theta(e)^{2}W(e)^2 \sqrt{q}.
  \end{equation}
  and
  \begin{equation}\label{Neq3}
  |N(e,le) -\theta(l)N(e,e)| \leq (1-1/l)\theta(e)^{2}W(e)^2 \sqrt{q}.
  \end{equation}
\end{lem}

\begin{proof}
Starting with the identity~(\ref{jac}) we use the fact that when $d_1=d_2=1$ (so that
  $\chi_{d_1} =\chi_{d_2}$ is the principal character of $\mathbb{F}_q$), then
  $\chi_{d_1}(1/c)(\chi_{d_1}\chi_{d_2})(a)J(\chi_{d_1},\chi_{d_2})=q-2$. For all other character
  pairs $(\chi_{d_1}, \chi_{d_2})$, as is well-known, this quantity has absolute value $\sqrt{q}$
  (at most). Because there are, for example, $\phi(d_1)$ characters $\chi_{d_1}$ of order $d_1$,
  when we take into account the implicit denominators $\phi(d_1)$ and the M\"{o}bius function
  within the integral notation, we obtain as an aggregate contribution to the right side
  of~(\ref{jac}) a quantity of absolute value at most $\sqrt{q}$ from each pair of
  \emph{square-free} divisors $d_1,d_2$ of $e$, except the pair (1,1). Hence
  $N(e,e) \geq \theta(e)^{2}\{q-2 -(W(e)^2-1) \sqrt{q}\}$, which yields~(\ref{Neq1}).

  Further, from~(\ref{jac}), since $\theta(le) =\theta(l)\theta(e)$,
  \[
    N(le,e)-\theta(l)N(e,e) = \theta(l)\theta(e)^2\int_{d_1|e}\int_{d_2|e}
      \chi_{ld_1}(1/c)(\chi_{ld_1}\chi_{d_2})(a)J(\chi_{ld_1},\chi_{d_2}).
  \]
  Hence,
  \[
    |N(le,e)-\theta(l)N(e,e)| \leq \theta(l)\theta(e)^2W(e)(W(le)- W(e)) \sqrt{q},
  \]
  which yields~(\ref{Neq2}), since $W(le) =2W(e)$. Similarly, (\ref{Neq3}) holds.
\end{proof}

We now complete the proof of Theorem~\ref{t:Cohen}.
\begin{proof}
  Assume $\delta>0$. From~(\ref{sieveeq2}) and Lemma~\ref{t:N}
  \begin{eqnarray*}
    N(q-1,q-1) & \geq & \theta(e)^2\left\{\delta(q-W(e)^2\sqrt{q})-
                        \sum_{i=1}^s 2\left(1-\frac{1}{p_i}\right)W(e)^2\sqrt{q}\right\} \\
               & =    & \delta \theta(e)^2\sqrt{q}
                        \left\{\sqrt{q}-W(e)^2-\left(\frac{2s-1}{\delta}+1\right)W(e)^2\right\}.
  \end{eqnarray*}
  The conclusion follows.
\end{proof}

\bibliographystyle{amsplain}
\bibliography{manuscript}

\end{document}